\documentclass[12pt,reqno]{amsart}
\usepackage{amscd}
\usepackage{graphicx}
\usepackage{mathrsfs}
\usepackage{cite}
\usepackage{ifpdf}
\usepackage{fancyhdr}
\usepackage{multirow}
\usepackage{array}
\usepackage{appendix}
\usepackage{longtable}
\usepackage{bm}
\usepackage{cancel}
\usepackage{makecell}
\usepackage{caption}

\usepackage{amssymb}
\usepackage{amsthm}
\usepackage{amsmath}
\usepackage{cancel}
\usepackage{tikz}
\usepackage{graphicx}
\usepackage{xcolor}
\usepackage{dsfont}
\usepackage{csquotes}
\usepackage{float}
\usepackage{soul}
\usepackage{extarrows}
\usepackage[shortlabels]{enumitem}
\usepackage[foot]{amsaddr}
\usepackage{doi}
\usepackage{bropd}
\usepackage{setspace}
\usepackage{ragged2e}
\usepackage{bookmark}
\usepackage{bm}
\usepackage{amsaddr}
\usepackage{subfig}
\usepackage{overpic}
\usepackage[a4paper,scale=0.76,twoside=false]{geometry}
\usepackage[skip=0pt plus1pt, indent=22pt]{parskip}
\allowdisplaybreaks
\usepackage{color,verbatim}
\usepackage[intlimits]{esint}
\usepackage[numbers,square,sort&compress]{natbib}
\usepackage{csquotes}
\usepackage{hyperref}
\usepackage[capitalise, nameinlink, noabbrev]{cleveref}
\hypersetup{colorlinks=true,
	linkcolor=blue,
	citecolor=red,
	allcolors=blue
}
\numberwithin{equation}{section}
\newtheorem{thm}{Theorem}[section]
\newtheorem{lem}[thm]{Lemma}
\newtheorem*{thmA}{Theorem A}
\newtheorem*{thmB}{Theorem B}

\newtheorem{prop}[thm]{Proposition}
\newtheorem{cor}[thm]{Corollary}

\newtheorem{rem}[thm]{Remark}

\newtheorem{assume}[thm]{Assumption}
\newtheorem{define}{Definition}[section]

\newcommand{\hrefemail}[1]{\href{mailto:#1}{#1}}

\usepackage[cal=euler,scr=boondox]{mathalpha}
\title[Helically symmetric solution and its free boundary]{Helically symmetric solution of 3D Euler equations with vorticity and its free boundary}
\author[Du]{$^{\dagger,\ddagger,1}$Lili Du}
\email{$^{1}$\hrefemail{ dulili@scu.edu.cn}}
\address{$^{\dagger}$School of Mathematical Sciences, Shenzhen University, 
Shenzhen, 518000, P.~R.~China}

\author[Ji]{$^{\ddagger,2}$Feng Ji}
\email{$^{2}$\hrefemail{jifeng\_math@126.com}, corresponding author}
\address{$^{\ddagger}$, Department of Mathematics, Sichuan University, Chengdu, 610000, P.~R.~China}

\begin{document}
	
\begin{abstract}
    This paper investigates an incompressible steady free boundary problem of Euler equations with helical symmetry in $3$ dimensions and with nontrivial vorticity. The velocity field of the fluid arises from the spiral of its velocity within a cross-section, whose global existence, uniqueness and well-posedness with fixed boundary were established by a series of brilliant works. A perplexing issue, untouched in the literature, concerns the free boundary problem with (partial) unknown domain boundary in this helically symmetric configuration. We address this gap through the analysis of the optimal regularity property of the scalar stream function as a minimizer in a semilinear minimal problem, establishing the $C^{0,1}$-regularity of the minimizer, and the $C^{1,\alpha}$-regularity of its free boundary. More specifically, the regularity results are obtained in arbitrary cross-sections through smooth helical transformation by virtue of variational method and the rule of "flatness implies $C^{1,\alpha}$".
\end{abstract}

\maketitle
\tableofcontents

\section{Introduction and main results}

In this article, we investigate the incompressible rotational fluid in a three-dimensional domain $\Omega\subset\mathbb{R}^3$, governed by the steady Euler equations,
\begin{equation} \label{Euler}
\begin{cases}
\nabla\cdot{\bm u}=0 \qquad\qquad\quad \text{in} \quad \Omega, \\
({\bm u}\cdot\nabla){\bm u} + \nabla p =0 \quad \text{in} \quad \Omega.
\end{cases}
\end{equation}
Here $\bm u=(u_1, u_2, u_3)$ denotes the velocity field of the fluid and $p$ denotes the scalar pressure. In addition, the following impermeability condition holds on the free boundary $\partial\Omega$, whose position is not known apriori, that
\begin{equation}
\bm u \cdot \bm n =0 \quad \text{on} \quad \partial\Omega
\end{equation}
with $\bm n$ the outer unit normal to $\partial\Omega$. Moreover, the vorticity
$$\bm w = \nabla\times\bm u$$
is not necessary a zero vector in $\Omega$.

Three dimensional fluids are hard to tackle even with free vorticity, and the simplest case of motion in three dimensions is the axisymmetric flows, which have attracted loads of mathematicians, in \cite{ACF82} and \cite{ACF83} for example. This investigation focuses the solutions of (\ref{Euler}) which are invariant under a helically symmetric transformation. Helical flows arise behind propellers and wind turbines, and helical vortices are important for helicopter rotor performance, see \cite{H82} and \cite{LD09}. Moreover, helical pipe flow has received considerable attention in engineering disciplines in \cite{G06},\cite{T90} and so on, which has been also studied as a computational model for a blood vessel with non-vanishing curvature and torsion in \cite{ZM98-1} and \cite{ZM98-2}. Figure \ref{F0} gives an illustration of rotor far-wake investigated in \cite{OS07}, and Figure \ref{F2} describes a helical pipe motivated by physiological applications.

\begin{figure}[!h]
	\includegraphics[width=100mm]{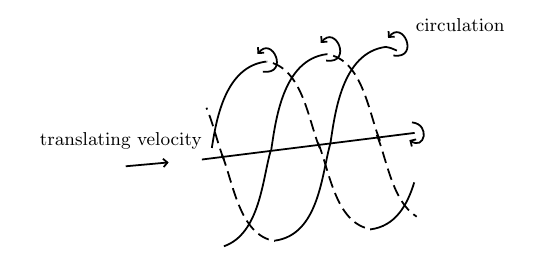}
	\caption{Sketch of the far-wake model.}
	\label{F0}
\end{figure}

\begin{figure}[!h]
	\includegraphics[width=80mm]{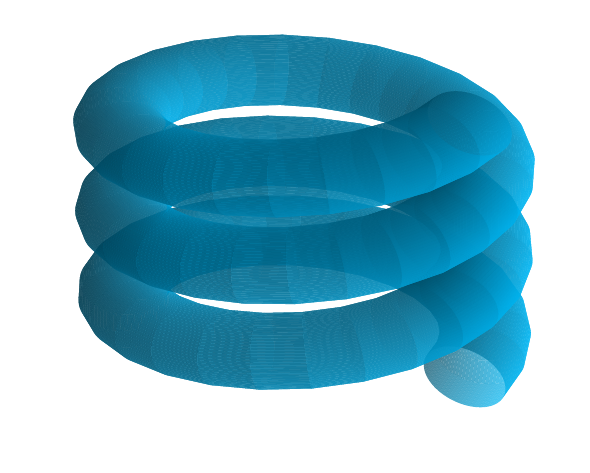}
	\caption{An example of helical pipe.}
	\label{F2}
\end{figure}

\subsection{Preliminary about helical symmetry}

Let $\bm x=(x,y,z)^T$. Fix a nonzero constant length scale $\kappa$ corresponding to the pitch of the helix. Define the helical symmetry group
$$G^\kappa:=\left\{ S_\rho^\kappa:\mathbb{R}^3\rightarrow\mathbb{R}^3 \ | \ \rho\in\mathbb{R} \right\},$$
where $S_\rho^\kappa$ denotes the screw transformation
\begin{equation*}
\begin{aligned}
S_\rho^\kappa \bm x &=S_\rho^\kappa(x,y,z)^T \\
&=(x\cos\rho+y\sin\rho, -x\sin\rho+y\cos\rho, z+\kappa\rho)^T \\
&=R_\rho (x,y,z)^T + (0,0,\kappa\rho)^T
\end{aligned}
\end{equation*}
with the rotation
$$R_\rho =
\begin{pmatrix}
\cos\rho & \sin\rho & \\
-\sin\rho & \cos\rho & \\
& & 1 \\
\end{pmatrix}.$$
In fact, $S_\rho^\kappa$ can be seen as the superposition of a simultaneous rotation around the $z$-axis with a translation along the $z$-axis, and $2\pi\kappa$ is the pitch of the helix. See Figure \ref{F1} for example.

\begin{figure}[!h]
	\includegraphics[width=80mm]{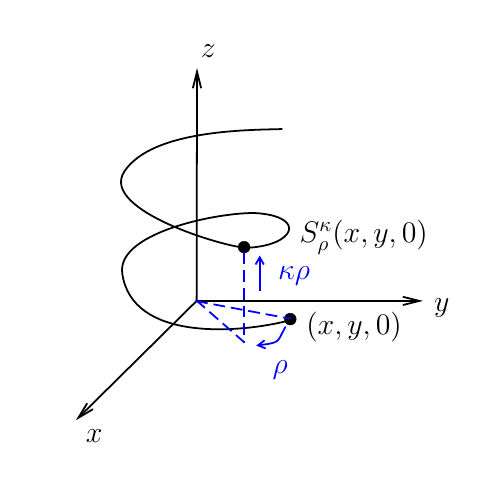}
	\caption{Helical transformation}
	\label{F1}
\end{figure}

We hope to solve the problem (\ref{Euler}) in a domain which is invariant under the screw transformation $S_\rho^\kappa$, namely, the \emph{helical domain}. We will also specify the definitions of \emph{helical function} and \emph{helical vector field}, since we expect to have helical solution $(\bm u, p)^T$ of (\ref{Euler}).

\begin{define} \label{hvf}
	(i) A domain $D\subset\mathbb{R}^3$ is called a ($\kappa$-)helical domain, provided that it satisfies
	$$S_\rho^\kappa D = D$$
	for any $\rho\in\mathbb{R}$ and some fixed $\kappa\neq0$.
	
	(ii)A function $f: \mathbb{R}^3\rightarrow\mathbb{R}$ is called a ($\kappa$-)helical function, provided that it satisfies
	$$f(S_\rho^\kappa\bm x) = f(\bm x)$$
	for any $\rho\in\mathbb{R}$ and some fixed $\kappa\neq0$.
	
	(iii)A vector field $\bm v=(v_1,v_2,v_3): \mathbb{R}^3\rightarrow\mathbb{R}^3$ is called a ($\kappa$-)helical vector field, provided that it satisfies
	$$\bm v(S_\rho^\kappa \bm x) = R_\rho \bm v (\bm x)$$
	for any $\rho\in\mathbb{R}$ and some fixed $\kappa\neq0$.
\end{define}

Observe that the group $G^\kappa$ has symmetry lines with tangent
$$\bm \xi_\kappa(\bm x) = (y, -x, \kappa)^T,$$
which plays an important role throughout the whole paper, since the helical functions are orthogonal to $\bm\xi_\kappa$ and the helical vector fields obey some rules corresponding to $\bm\xi_\kappa$ as pointed out in \cite{ET09}. In order not to disturb the flow of the paper we put the proof into Appendix \ref{appe1}.

\begin{lem} \label{lem1.1}
	Suppose $f$ is a differentiable function and $\bm v=(v_1,v_2,v_3)$ is a differentiable vector field in the helical domain $D\subset\mathbb{R}^3$. Then the function $f$ is helical if and only if
	$$\nabla f\cdot\bm\xi_\kappa=0,$$
	and the vector field $\bm v$ is helical if and only if
	\begin{equation*}
	\begin{cases}
	\nabla v_1 \cdot\bm\xi_\kappa = v_2, \\
	\nabla v_2 \cdot\bm\xi_\kappa = -v_1, \\
	\nabla v_3 \cdot\bm\xi_\kappa = 0.
	\end{cases}
	\end{equation*}
\end{lem}

Summarily, the flows with helical symmetry fall within a class of "two-and-a-half" dimensional flows, namely flows in a three-dimensional domain with certain spatial symmetry, but opposed to the axisymmetric case since there is no orthogonal coordinates such that one of them runs along the symmetry lines.

\subsection{Physical setting of the free boundary problem}

Utilizing helical symmetry, the 3D Euler problem (\ref{Euler}) can be reduced to a 2D vorticity equation as follows. We first give an orthogonal condition of $\bm u$ to the symmetry lines $\bm\xi_\kappa$ of the group $G^\kappa$, such that the vorticity $\bm w$ is directed along the symmetry lines to avoid the vorticity stretching. In this sense the orthogonal condition is somewhat similar to the assumption that the azimuthal component of the velocity is zero in the axisymmetric setting.

\begin{assume} (Orthogonal condition) \label{assu}
	Throughout this article we will assume that the velocity field $\bm u$ of the Euler equations (\ref{Euler}) are orthogonal to the helices, namely,
	\begin{equation} \label{ortho}
	\bm u\cdot\bm \xi_\kappa = yu_1 - xu_2 + \kappa u_3 =0.
	\end{equation}
\end{assume}

\begin{rem}
	In this paper we consider the steady Euler flow. Nevertheless, suppose that the initial velocity $\bm u_0$ is the helical vector field that gives rise to the helical solution $\bm u$ of three-dimensional unsteady Euler equations, then $\bm u_0\cdot\bm\xi_\kappa=0$ implies that $\bm u\cdot\bm\xi_\kappa=0$ at any time $t$. See \cite{ET09}, Corollary 2.8.
\end{rem}

\subsubsection{Dimensionality reduction and the stream function}

Under Assumption \ref{assu}, by virtue of $\nabla u_3\cdot\bm\xi_\kappa=0$ we can rewrite the first equation in (\ref{Euler}) that
\begin{equation*}
\begin{aligned}
0 &= \frac{\partial u_1}{\partial x} + \frac{\partial u_2}{\partial y} + \frac{1}{\kappa}\left( -y\frac{\partial u_3}{\partial x} + x\frac{\partial u_3}{\partial y} \right) \\
&=\frac{\partial u_1}{\partial x} + \frac{\partial u_2}{\partial y} + \frac{1}{\kappa^2}\left[ -y\frac{\partial (-yu_1+xu_2)}{\partial x} + x\frac{\partial (-yu_1+xu_2)}{\partial y} \right] \\
&= \frac{1}{\kappa^2}\frac{\partial}{\partial x}\left[ (\kappa^2+y^2)u_1 - xyu_2 \right] + \frac{1}{\kappa^2}\frac{\partial}{\partial y}\left[ (\kappa^2+x^2)u_2 - xyu_1 \right].
\end{aligned}
\end{equation*}
Hence we can consider the problem in a two-dimensional domain
$$\Omega_{\rho_0}=\Omega\cap\{z=\rho_0\}.$$
Without loss of generality we focus on $D_0$, namely,
$$\Omega_0=\left\{ (x,y)^T \ | \ (x,y,0)\in \Omega \right\}.$$
The fluid domain $\Omega$ and the free boundary $\partial\Omega$ can be inherited from $\Omega_0$ and $\partial\Omega_0$, see Figure \ref{F3} for an illustration.
\begin{figure}[!h]
	\includegraphics[width=90mm]{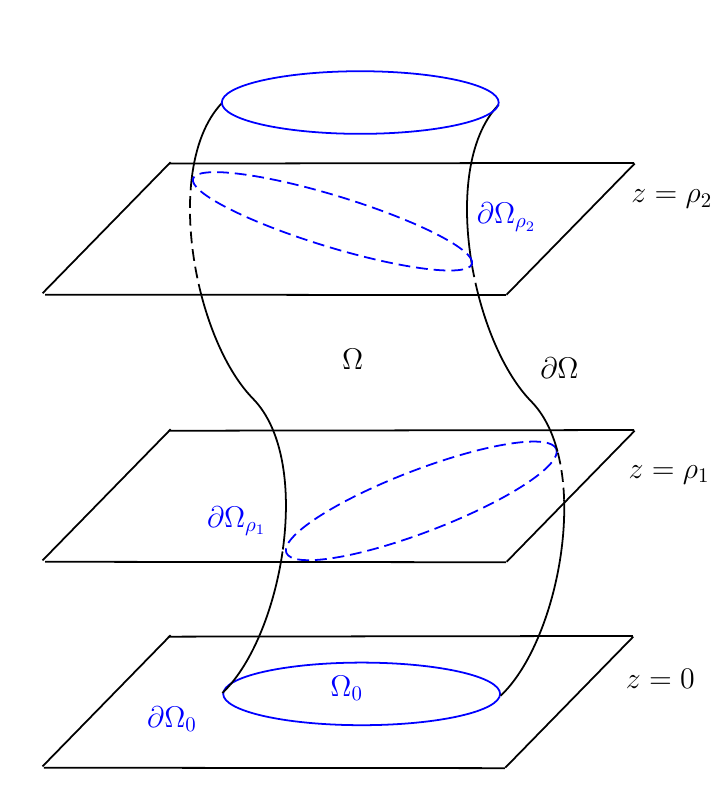}
	\caption{Dimensionality reduction of $\Omega$.}
	\label{F3}
\end{figure}

For notational simplicity denote
$$u_i^0(x,y)=u_i(x,y,0) \quad \text{for} \quad i=1,2,3.$$
Thus there exists a stream function $\psi(x,y):\mathbb{R}^2\rightarrow\mathbb{R}$ in $\Omega_0$ satisfying
$$\frac{\partial\psi}{\partial x}=\frac{1}{\kappa^2}\left[-(\kappa^2+x^2)u_2^0 + xyu_1^0\right], \quad \frac{\partial\psi}{\partial y}=\frac{1}{\kappa^2}\left[(\kappa^2+y^2)u_1^0 - xyu_2^0\right].$$
Denote
\begin{equation} \label{K}
K(x,y)= \frac{1}{\kappa^2+x^2+y^2}
\begin{pmatrix}
\kappa^2+y^2 & -xy \\
-xy & \kappa^2+x^2 \\
\end{pmatrix},
\end{equation}
we have that
\begin{equation} \label{re1}
\begin{pmatrix}
u_1^0 \\
u_2^0
\end{pmatrix}
=
\begin{pmatrix}
0 & 1 \\
-1& 0 \\
\end{pmatrix}
K(x,y)
\begin{pmatrix}
\frac{\partial\psi}{\partial x} \\
\frac{\partial\psi}{\partial y} \\
\end{pmatrix}
\ \text{and} \
\begin{pmatrix}
\frac{\partial\psi}{\partial x} \\
\frac{\partial\psi}{\partial y} \\
\end{pmatrix}
=
K^{-1}(x,y)
\begin{pmatrix}
0 & -1 \\
1 & 0 \\
\end{pmatrix}
\begin{pmatrix}
u_1^0 \\
u_2^0
\end{pmatrix},
\end{equation}
where
$$K^{-1}(x,y)=\frac{1}{\kappa^2}
\begin{pmatrix}
\kappa^2+x^2 & xy \\
xy & \kappa^2+y^2 \\
\end{pmatrix}
$$
is the inverse matrix of $K(x,y)$.

\subsubsection{Formulation of the free boundary}

In the free boundary problem, (a portion of) the domain boundary $\partial\Omega$ remains unknown, and is determined by the solution. We propose to parametrize the free boundary $\partial\Omega_0$ within the cross-section $\{z=0\}$ via a scalar potential function $\psi:\mathbb{R}^2\rightarrow\mathbb{R}$, establishing a functional representation of the interface geometry.

Suppose $\Omega$ is the helical domain with smooth enough boundary. Let
$$\bm n(\bm x)=(n_1(\bm x),n_2(\bm x),n_3(\bm x))^T\in\mathbb{R}^3$$
be the unit outer normal to $\partial\Omega$ at $\bm x=(x,y,z)^T\in\partial\Omega$, and
$$\bm\nu(X)=(\nu_1(X),\nu_2(X))^T\in\mathbb{R}^2$$
be the unit vector tangent to $\partial\Omega_0$ at $X=(x,y)\in\partial\Omega_0$. Since $\Omega$ is a helical domain, all the points of the form $S_\rho^\kappa((X,0))$ belongs to $\partial\Omega$ for any $\rho>0$ and every $X\in\partial\Omega_0$. This observation indicates that $\bm\xi_\kappa=(y,-x,\kappa)$ is tangent to $\partial\Omega$, that is,
\begin{equation} \label{eq2}
\bm\xi_\kappa\cdot\bm n=0.
\end{equation}
Hence we can denote
$$\bm n=\left( n_1,n_2,-\frac{y}{\kappa}n_1 + \frac{x}{\kappa}n_2 \right).$$
Combining this with the impermeability condition $\bm u\cdot\bm n=0$ on $\partial\Omega$, we obtain
\begin{equation*}
\begin{aligned}
0 &= u_1n_1 + u_2n_2 + \left( -\frac{y}{\kappa}u_1 + \frac{x}{\kappa}u_2 \right)\left( -\frac{y}{\kappa}n_1 + \frac{x}{\kappa}n_2 \right) \\
&=\frac{1}{\kappa^2}
\begin{pmatrix}
u_1 & u_2 \\
\end{pmatrix}
\begin{pmatrix}
\kappa^2+y^2 & -xy \\
-xy & \kappa^2+x^2 \\
\end{pmatrix}
\begin{pmatrix}
n_1 \\
n_2
\end{pmatrix}.
\end{aligned}
\end{equation*}
Thus, $\begin{pmatrix}
n_1 \\
n_2 \\
\end{pmatrix}$ is proportional to the vector $\begin{pmatrix}
-(\kappa^2+x^2)u_2 + xyu_1 \\
-xyu_2 + (\kappa^2+y^2)u_1 \\
\end{pmatrix}$.

On the other hand, since $(\nu_1,\nu_2,0)^T$ is tangent to $\partial\Omega$ at $(x,y,0)^T$, we have
$$(\nu_1,\nu_2,0)\cdot\bm n = \nu_1n_1+\nu_2n_2 = 0.$$
This implies that $\bm\nu$ is proportional to $\begin{pmatrix}
(\kappa^2+y^2)u_1^0 - xyu_2^0 \\
(\kappa^2+x^2)u_2^0 - xyu_1^0 \\
\end{pmatrix}$. Now compute
\begin{equation*}
\begin{aligned}
\nabla\psi\cdot\bm\nu &= \frac{\partial\psi}{\partial x}\nu_1 + \frac{\partial\psi}{\partial y}\nu_2 \\
&=C[xyu_1^0-(\kappa^2+x^2)u_2^0][(\kappa^2+y^2)u_1^0-xyu_2^0] \\
&\quad +C[(\kappa^2+y^2)u_1^0-xyu_2^0][(\kappa^2+x^2)u_2^0-xyu_1^0] \\
&=0
\end{aligned}
\end{equation*}
for some $C>0$. Hence $\psi$ is defined up to a constant on the free boundary $\partial\Omega_0$, and without loss of generality, we choose
$$\psi\equiv0 \quad \text{on} \quad \partial\Omega_0.$$

\subsubsection{Free boundary problem}

We now derive the two-dimensional governing equation in $\{\psi>0\}$ and the free boundary condition on $\partial\{\psi>0\}$. Direct calculation shows that $\bm w$ is proportional to $\bm\xi_\kappa$, namely,
$$\bm w = \nabla\times\bm u =\frac{1}{\kappa}\left( \frac{\partial u_2}{\partial x} - \frac{\partial u_1}{\partial y} \right)\bm\xi_\kappa.$$
Take the curl of both sides of the second equation in (\ref{Euler}), we obtain
\begin{equation} \label{curl}
(\bm u\cdot\nabla)\bm w + (\bm w\cdot\nabla)\bm u =0.
\end{equation}
Define
$$W(x,y):=\frac{\partial u_2^0}{\partial x} - \frac{\partial u_1^0}{\partial y}$$
on $\{z=0\}$, and then (\ref{curl}) implies that $(\bm u\cdot\nabla)W + \frac{1}{\kappa}W\nabla u_3\cdot\bm\xi_\kappa=0$, which together with the fact that $\nabla u_3\cdot\bm\xi_\kappa=0$ leads to
$$\bm u \cdot \nabla W =0.$$
Consequently $W$ is invariant along the streamline, and we can write
$$W = f(\psi):\mathbb{R}\rightarrow\mathbb{R}.$$
Recall the definition of $W(x,y)$ and the formulation (\ref{re1}), we obtain the uniformly elliptic equation
$$- div \left( K(x,y)\nabla\psi \right) = f(\psi) \quad \text{in} \quad \{\psi>0\}.$$

On the other hand, the \emph{Bernoulli's law} shows that
$$\frac12\bm u^2 + p = \mathcal{B}$$
on each streamline for some uniform constant $\mathcal{B}$. Taking $z=0$ and plugging (\ref{ortho})(\ref{re1}) into it, and noticing that
$$p=p_0 \quad \text{on} \quad \partial\Omega$$
with $p_0$ the atmospheric pressure, we obtain
\begin{equation*}
\left<K(x,y)\nabla\psi,\nabla\psi\right> = 2(\mathcal{B}-p_0) \quad \text{on} \quad \partial\{\psi>0\},
\end{equation*}
where $\left<K(x,y)\nabla\psi,\nabla\psi\right>$ denotes the inner product of $K(x,y)\nabla\psi$ and $\nabla\psi$. Hence $\psi$ is the solution to the two-dimensional semilinear free boundary problem
\begin{equation}\label{eq1}
\begin{cases}
div (K(x,y)\nabla\psi) = f(\psi) \qquad\qquad\quad \text{in} \quad \{\psi>0\}, \\
\left<K(x,y)\nabla\psi,\nabla\psi\right> = \lambda^2 \qquad\qquad\quad\ \text{on} \quad \partial\{\psi>0\}
\end{cases}
\end{equation}
with $K(x,y)$ defined as in (\ref{K}) and $\lambda^2=2(\mathcal{B}-p_0)>0$.

\subsubsection{Recovering to three-dimensional solution}

By virtue of (\ref{re1}), the solution $\psi$ to (\ref{eq1}) gives the explicit form of $u_1^0(x,y)=u_1(x,y,0)$ and $u_2^0(x,y)=u_2(x,y,0)$. Recall the orthogonal assumption (\ref{ortho}),
$$u_3(x,y,0)=\frac{1}{\kappa}(-yu_1(x,y,0)+xu_2(x,y,0)).$$
Hence we recover the $u_3$ component on $\Omega_0\times\{0\}$. The velocity field $\bm u$ in $\Omega$ follows from Definition \ref{hvf} that
$$\bm u(x,y,z) = R_{\frac{z}{\kappa}}\bm u (S_{-\frac{z}{\kappa}}^\kappa(x,y,z)).$$
Meanwhile, we can solve $p$ through the PDE
\begin{equation*}
\begin{cases}
\Delta p = div\left( (\bm u\cdot\nabla)\bm u \right) \quad \text{in} \quad \Omega, \\
p=p_0 \qquad\qquad\qquad\quad\, \text{on} \quad \partial\Omega
\end{cases}
\end{equation*}
with some constant atmosphere pressure $p_0$. Hence, $(\bm u,p)$ is the helical solution to the Euler equations (\ref{Euler}) which is deduced from the solution $\psi$ to the 2-dimensional free boundary problem (\ref{eq1}). It is remarkable that the regularity of the free boundary $\partial\Omega$ can be inherited from $\partial\Omega_0$ through a smooth transformation, and we focus only on its free boundary within the cross-section $\{z=0\}$.

More geometric properties of helical flows are referred to \cite{AB14} and \cite{ET09}.

\subsection{Variational approach and main theorems}

In this subsection, we will understand the helical solution of (\ref{eq1}) by way of the variational approach, and reformulate the problem as an energy minimization framework. The analysis of such energy minimizing problem, associated with the investigation of the free boundary, was launched by Alt and Caffarelli in 1981. The celebrated A-C functional
\begin{equation} \label{ac}
\mathcal{J}_{\rm AC}(\psi,D)=\int_D \left( |\nabla\psi|^2 + \lambda^2\chi_{\{\psi>0\}} \right)dX
\end{equation}
has been a focal point of scholarly exploration ever since its introduction in \cite{AC81} and given rise to a substantial body of literature and sparked research endeavors. They proved that all the free boundary points are $C^{1,\alpha}$ regular in two dimensions for the Laplacian operator and $f\equiv0$. After that, Caffarelli applied a general strategy by a series works of \cite{C87}, \cite{C88} and \cite{C89} to attack the regularity of free boundary based on a powerful monotonicity formula, Harnack principles and families of continuous perturbation to reveal the criterion that "flatness implies $C^{1,\alpha}$" in any dimensions. The revelation marked the beginning of a burgeoning and fertile area of the research of the free boundary points, and the theories are extended to more complex governing operators in \cite{ACF84}, \cite{DP05} and \cite{OY90}. Recently in 2011, De Silva has utilized a new methodology of partial boundary Harnack and linearization to obtain an improvement of flatness approach for viscosity solutions in \cite{S11}, which was applied to study more general free boundary problems in \cite{SFS14}, \cite{SFS15} and \cite{SFS16}.

In this paper we consider the functional
$$\mathcal{J}_h(\psi,D):=\int_D \left( \left<K(x,y)\nabla\psi,\nabla\psi\right> - 2F(\psi) +\lambda^2\chi_{\{\psi>0\}} \right)dX$$
with the constant $\lambda>0$ in the admissible set
\begin{equation} \label{A}
\mathcal{A}:=\left\{ \psi\in W^{1,2}(D) \ | \ \psi\geq0 \ \text{on} \ S \right\},
\end{equation}
where $D$ is a domain in $\mathbb{R}^2$, $F(s)=\int_0^s f(s)ds$ is the primitive for some given function $f(s)\in C^{1,\beta}(\mathbb{R})$, $\chi_{\{\psi>0\}}$ denotes the characterization function of the set $\{\psi>0\}$, $dX=dxdy$ and $S\subset\partial D$ is a portion of fixed boundary. Moreover, we give further assumptions about $f$ that
\begin{equation}\label{assu2}
0\leq f(s) \leq F_0 \quad \text{for} \quad s\leq0 \qquad \text{and} \qquad -F_0\leq f'(s) \leq 0 \quad \text{for} \quad s\in\mathbb{R}
\end{equation}
for some given constant $F_0>0$. This assumption indicates that $F''(s)\leq0$ for all $s\in\mathbb{R}$, namely, for any $p,q\in\mathbb{R}$,
$$F(q)-F(p)\geq F'(q)(q-p).$$
Observe that if $X_0\in\partial\{\psi>0\}\subset D$, then $f(u(X_0))=f(0)\geq0$ and this implies that the vorticity is non-negative on any free boundary points. Under this setting a standard variational approach, which is contained in Appendix \ref{appe2}, shows that if $\psi$ minimizes $\mathcal{J}_h(\psi,D)$ among all the functions in $\mathcal{A}$, then $\psi$ satisfies
\begin{equation*}
\begin{cases}
- div (K(x,y)\nabla\psi) = f(\psi) \qquad\qquad\quad \text{in} \quad \{\psi>0\}\cap D, \\
\left<K(x,y)\nabla\psi,\nabla\psi\right> = \lambda^2 \qquad\qquad\qquad\, \text{on} \quad \partial\{\psi>0\}\cap D.
\end{cases}
\end{equation*}

\begin{rem}
	We give an example of $f(s)\in C^{1,\beta}(\mathbb{R})$ satisfying condition (\ref{assu2}),
	$$f(s)=\frac{1}{1+e^{4s}}-\frac12.$$
	It is straightforward to check that $0\leq f(s)\leq\frac12$ when $s\leq0$, and
	$$f'(s)=-\frac{4e^{4s}}{(1+e^{4s})^2}\in[-1,0]$$
	for any $s\in\mathbb{R}$.
\end{rem}

This paper focuses on the regularity theory of the solution $\psi$ and the free boundary $\partial\{\psi>0\}$, and we give the following definition of \emph{local minimizer} in a ball $B_R(X_0)$ centered at the free boundary point $X_0\in\partial\{\psi>0\}$.

\begin{define}
	We say the helical function $\psi\in W^{1,2}(B_R(X_0))$ with $\psi\geq0$ is a local minimizer of $\mathcal{J}_h$ in $B_R(X_0)\subset D$, if and only if
	$$\mathcal{J}_h(\psi,B_R(X_0))\leq\mathcal{J}_h(\tilde{\psi},B_R(X_0))$$
	for all functions $\tilde{\psi}\in W^{1,2}(B_R(X_0))$ with $\psi=\tilde{\psi}$ on $\partial B_R(X_0)$.
\end{define}

We have two main results. The first one concerns the optimal regularity and the non-degeneracy of the local minimizer $\psi$.

\begin{thmA}
	Suppose $\psi$ is a local minimizer of $\mathcal{J}_h$ in $B_R(X_0)$, where $X_0\in\partial\{\psi>0\}$ is a free boundary point. Then the following properties hold:
	
	(i) Lipschitz continuity: $\psi\in C^{0,1}(B_R(X_0))$.
	
	(ii) Non-degeneracy: There is a constant $c_0=c_0(\lambda)>0$ such that
	\begin{equation*}
	\left( \fint_{\partial B_r(X_0)} \psi^2 dX \right)^{1/2} \geq c_0 r \quad \text{for} \quad 0<r<R/2.
	\end{equation*}
\end{thmA}

The whole Section 2 is dedicated to the proof of Theorem A by virtue of the property of the minimizer.

The second result shows that the free boundary $\partial\{\psi>0\}$ is $C^{1,\alpha}$ regular.

\begin{thmB}
	Suppose $\psi$ is a local minimizer of $\mathcal{J}_h$ in $B_R(X_0)$, where $X_0\in\partial\{\psi>0\}$ is a free boundary point. Then there exists $R_0=R_0(\lambda,F_0,X_0)$ such that $\partial\{\psi>0\}\cap B_{R_0}(X_0)$ is a $C^{1,\alpha}$ graph for some $\alpha\in(0,1)$.
\end{thmB}

The outline of the proof of Theorem B is stated in the beginning of Section 3.

\begin{cor}
	The 3-dimensional free boundary $\partial\Omega$ in the helical free boundary problem (\ref{Euler}) is a $C^{1,\alpha}$ surface for some $\alpha\in(0,1)$.
\end{cor}
\begin{proof}
	Theorem B indicates that after the smooth helical transformation, the boundary of the 3-dimensional helical domain $\Big\{ (x,y,z) \ | \ S_{-\frac{z}{\kappa}}^\kappa(x,y,z)\in\{\psi>0\}\times\{z=0\} \Big\}$ is locally a $C^{1,\alpha}$ surface. In fact, the orthogonal condition (\ref{ortho}) that $\bm u\cdot\bm\xi_\kappa=0$ and the relationship (\ref{eq2}) that $\bm\xi_\kappa\cdot\bm n=0$ imply the relationship
	$$\bm n(x,y,z)=\frac{\bm u(x,y,z)\times\bm\xi_\kappa}{|\bm u(x,y,z)\times\bm\xi_\kappa|}$$
	with
	$$\bm u(x,y,z)=R_{\frac{z}{\kappa}}\bm u (S_{-\frac{z}{\kappa}}^\kappa(x,y,z))=R_{\frac{z}{\kappa}}\bm u (\bar{x},\bar{y},0),$$
	where
	$$x=\bar{x}\cos\frac{z}{\kappa} + \bar{y}\sin\frac{z}{\kappa}, \quad y=-\bar{x}\sin\frac{z}{\kappa}+\bar{y}\cos\frac{z}{\kappa}$$
	and
	\begin{equation*}
	\bm u(\bar{x},\bar{y},0)=\frac{1}{\kappa^2+\bar{x}^2+\bar{y}^2}
	\begin{pmatrix}
	 -\bar{x}\bar{y}\frac{\partial\psi}{\partial x} + (\kappa^2+\bar{x}^2)\frac{\partial\psi}{\partial y}\\
	 -(\kappa^2+\bar{y}^2)\frac{\partial\psi}{\partial x} + \bar{x}\bar{y}\frac{\partial\psi}{\partial y}\\
	-\kappa \bar{x}\frac{\partial\psi}{\partial x} + -\kappa \bar{y}\frac{\partial\psi}{\partial y} \\
	\end{pmatrix}.
	\end{equation*}
	The boundary estimate in elliptic theory gives that $\psi(x,y)\in C^{1,\alpha}(\mathbb{R}^2\cap\overline{\{\psi>0\})}$, hence $\bm u$ is locally a $C^{\alpha}$ vector field in $\overline{\Omega}$. Notice that $|\bm u(x,y,z)\times\bm\xi_\kappa|$ cannot degenerate since $|\bm u|>0$ and $|\bm\xi_\kappa|\geq|\kappa|>0$ on $\partial\Omega$, hence $\bm n=\frac{\bm u(x,y,z)\times\bm\xi_\kappa}{|\bm u(x,y,z)\times\bm\xi_\kappa|}\in C^{\alpha}(\mathbb{R}^3;\mathbb{R}^3)$, which ensures that $\partial\Omega$ is locally a $C^{1,\alpha}$ surface.
\end{proof}

\begin{rem}
	Our results can be extended to more general minimizing problems of the form
	$$\min_{\psi\in\mathcal{A}} \mathcal{J}_{g}=\int_D \left( \left<A(X)\nabla\psi,\nabla\psi\right> - 2F(\psi) +Q(X)\chi_{\{\psi>0\}} \right)dX,$$
	where $A(X)=(a_{ij}(X))_{2\times2}$ is a symmetric matrix, $a_{ij}(X)$ and $Q(X)$ are H\"older continuous functions satisfying
	$$\frac{1}{C}|\bm \nu|^2\leq \Sigma_{i,j} a_{ij}\nu_i\nu_j \leq C|\bm\nu|^2\quad \text{for any} \quad \bm\nu=(\nu_1,\nu_2)\in\mathbb{R}^2 \quad \text{and some} \quad C>0$$
	and non-degenerate condition
	$$Q(X)\geq c_0>0.$$
\end{rem}

\begin{rem}
	The higher regularity of the free boundary $\partial\{\psi>0\}$ can be achieved through a bootstrap argument utilizing the Schauder estimate in the elliptic theory, provided that $f(\psi)$ is smooth enough. In fact, we have $K(x,y)\in C^\infty(\mathbb{R}^2)$ and if we assume additionally that $f\in C^\infty(\mathbb{R})$, then the free boundary $\partial\{\psi>0\}$ is $C^\infty$, and after the helical transformation the free boundary in three dimensions obtains the $C^\infty$-regularity.
\end{rem}

It is important to highlight that working with a functional with variable coefficients, instead of constant ones, leads to some differences from established results concerning Laplacian operators. The main difficulty arises from the lack of the Weiss monotonicity formula for the uniformly semilinear elliptic equation in the blow-up analysis. Through the change of coordinates, we freeze the variable coefficients to reduce the non-constant-coefficients case (as a minimizer) to the constant-coefficients one (as an almost-minimizer), and this adaption permits the blow-up analysis for almost-minimizers. Our results can be applied in the regularity research on the free boundary problem of 3D helical Euler equations, since we have established a dimensional reduction framework that recovers two-dimensional solutions within three-dimensional configurations. 

\section{Regularity of the stream function}

At the beginning of this section, we first give a fundamental properties of the local minimizer $\psi$, which is substantial in proving Theorem A.

\begin{prop}
	Suppose $\psi$ is a local minimizer of $\mathcal{J}_h$ in $B_R(X_0)$, where $X_0\in\partial\{\psi>0\}$ is a free boundary point. Then $\psi$ satisfies
	$$ div \left( K(x,y)\nabla\psi \right) + f(\psi) \geq0 \quad \text{in} \quad B_R(X_0)$$
	in the weak sense.
\end{prop}

\begin{proof}
	Let $\phi\in C_0^\infty(B_R(X_0))$ be any non-negative function and $\epsilon>0$. Since $\psi$ is a local minimizer in $B_R(X_0)$, we obtain
	\begin{equation*}
	\begin{aligned}
	0 &\leq \mathcal{J}_h(\psi-\epsilon\phi,B_R(X_0)) - \mathcal{J}_h(\psi,B_R(X_0)) \\
	&\leq \int_{B_R(X_0)} \left( -2\epsilon\left<K(x,y)\nabla\psi,\nabla\phi\right> -2(F(\psi-\epsilon\phi)-F(\psi)) \right)dX + o(\epsilon),
	\end{aligned}
	\end{equation*}
	where $o(\epsilon)$ denotes the infinitesimal of the higher order than the quantity $\epsilon$. Notice that we have used the facts that $K(x,y)$ is symmetric and $\{\psi-\epsilon\phi>0\}\subset\{\psi>0\}$. Dividing both sides of the inequality by $\epsilon$, we obtain
	$$\int_{B_R(X_0)} \left( \left<K(x,y)\nabla\psi,\nabla\phi\right> -f(\psi)\phi \right)dX \leq0,$$
	which gives the desired result.
\end{proof}

\subsection{Lipschitz regularity}

In this subsection we will prove that the local minimizer $\psi$ of $\mathcal{J}_h$ is Lipschitz continuous. The original argument was proposed by Alt and Caffarelli in \cite{AC81} for the minimum problem with Laplacian operator. The method was fully developed for two-phase Laplacian minimum problem in \cite{ACF84-2}, for quasilinear operators in \cite{ACF84} and for nonlinear operators in \cite{OY90}. We sketch the proof here for the sake of completeness.

\begin{proof}[Proof of Theorem A (i)]
	Denote $\rho(X):=dist(X,\{\psi=0\})$ with $X=(x,y)\in B_R(X_0)$. We first claim that there is a uniform constant $C=C(X_0,F_0,\kappa,r,R,\lambda)$ independent of $X$ such that for any $r<R$,
	$$\frac{\psi(X)}{\rho(X)}\leq C \quad \text{in} \quad B_r(X_0),$$
	which suffices to show the desired conclusion. To this end, let $\rho_0=R-r$ and $X_1=(x_1,y_1)\in B_r(X_0)\cap\{\psi>0\}$ with $\rho(X_1)<\rho_0$. Assume that
	$$\frac{\psi(X_1)}{\rho(X_1)}\geq M$$
	and consider the scaled sequence
	$$\phi_\rho(X):=\frac{\psi(X_1+\rho X)}{\rho} \quad \text{with} \quad \rho=\rho(X_1).$$
	Then $\phi_\rho$ satisfies
	$$div( K(X_1+\rho X)\nabla\phi_\rho ) + \rho f(\rho\phi_\rho) = 0 \quad \text{in} \quad B_1(0).$$
	The Harnack inequality in \cite{GT77} gives that
	$$\phi_\rho(X)\geq c\phi_\rho(0)-CF_0\geq cM-CF_0 \quad \text{in} \quad B_{3/4}(0)$$
	with $c$ and $C$ depend only on $X_0$ and $F_0$.
	
	On the other hand, let $\Psi(X)$ be the solution to
	\begin{equation*}
	\begin{cases}
	div(K(X_1+\rho X)\nabla\Psi) + \rho f(\rho\Psi) = 0 \quad \text{in} \quad B_1(X_2), \\
	\Psi=\phi_\rho \qquad\qquad\qquad\qquad\qquad\qquad\quad\ \text{on} \quad \partial B_1(X_2),
	\end{cases}
	\end{equation*}
	where $X_2=(x_2,y_2)\in\partial B_1(0)\cap\{\phi_\rho=0\}$. Since $div(K(X_1+\rho X)\nabla\phi_\rho) + \rho f(\rho\phi_\rho) \geq 0$ in $B_1(X_2)$, we can deduce from the maximum principle that
	$$\Psi(X)\geq\phi_\rho(X) \quad \text{in} \quad B_1(X_2).$$
	Furthermore, we have
	\begin{align*}
	0 &\leq \int_{B_1(X_2)} \left[ \left<K(X_1+\rho X)\nabla\Psi,\nabla\Psi\right> - 2F(\rho\Psi) +\lambda^2\chi_{\{\Psi>0\}} \right]dX \\
	&\quad - \int_{B_1(X_2)} \left[ \left<K(X_1+\rho X)\nabla\phi_\rho,\nabla\phi_\rho\right> - 2F(\rho\phi_\rho) +\lambda^2\chi_{\{\phi_\rho>0\}} \right]dX \\
	&=\int_{B_1(X_2)} \left[ -\left<K(X_1+\rho X)\nabla(\Psi-\phi_\rho),\nabla(\Psi-\phi_\rho)\right> \right. \\
	&\quad \left. + 2\left<K(X_1+\rho X)\nabla\Psi,\nabla(\Psi-\phi_\rho)\right> \right]dX +\int_{B_1(X_2)} -2(F(\rho\Psi)-F(\rho\phi_\rho))dX \\
	&\quad +\int_{B_1(X_2)} \lambda^2 \left( \chi_{\{\Psi>0\}}-\chi_{\{\phi_\rho>0\}} \right)dX \\
	&\leq \int_{B_1(X_2)} -\left<K(X_1+\rho X)\nabla(\Psi-\phi_\rho),\nabla(\Psi-\phi_\rho)\right>dX + \int_{B_1(X_2)} \lambda^2\chi_{\{\phi_\rho=0\}}dX,
	\end{align*}
	which implies that
	\begin{equation*}
	\begin{aligned}
	\int_{B_1(X_2)}|\nabla(\Psi-\phi_\rho)|^2dX &\leq C(\kappa,r,R)\int_{B_1(X_2)}\left<K(X_1+\rho X)\nabla(\Psi-\phi_\rho),\nabla(\Psi-\phi_\rho)\right>dX \\
	&\leq C(\kappa,r,R)\lambda^2\int_{B_1(X_2)}\chi_{\{\phi_\rho=0\}}dX.
	\end{aligned}
	\end{equation*}
	Hence we get the upper bound of the $L^2$-norm of $\nabla(\Psi-\phi_\rho)$. With the help of the auxiliary function $\phi(X)=(cM-CF_0)(e^{-\mu|X-X_2|^2}-e^{-\mu})$ for large enough $\mu$, we obtain
	$$\Psi(X)\geq (cM-CF_0)(1-|X-X_2|^2) \quad \text{in} \quad B_1(X_2).$$
	Proceeding as in Lemma 2.2 in \cite{ACF84}, we get
	$$(cM-CF_0)^2 \leq C(X_0,F_0,\kappa,r,R,\lambda).$$
	Thus
	$$M\leq C(X_0,F_0,\kappa,r,R,\lambda).$$
	Utilizing Harnack inequality and the finite covering property in $B_r(X^0)$, we have
	$$\frac{\psi(X)}{\rho(X)}\leq C(X_0,F_0,\kappa,r,R,\lambda) \quad \text{in} \quad B_r(X_0).$$
	This completes the proof.
\end{proof}

\subsection{Non-degeneracy}

In this subsection we establish the non-degeneracy of the minimizer $\psi$. The methodology comes from Alt and Caffarelli in \cite{AC81}, and the readers can refer to the lecture notes \cite{V23} for more versatile approaches adapted to almost minimizers.

\begin{proof}[Proof of Theorem A (ii)]
	It suffices to prove that for any $\delta\in(0,1)$, there exists a constant $c^*=c^*(\delta)$ such that if the inequality
	$$\frac1r \left( \fint_{\partial B_r(X_0)}\psi^2 dX \right)^{1/2}\leq c^*$$
	holds for $B_r(X)\subset B_R(X^0)$ with any small $r<R$ small enough, then it implies that
	$$\psi(X)\equiv0 \quad \text{in} \quad B_{\delta r}(X).$$
	
	Let $X_1\in B_R(X_0)$ and define
	$$\phi_r(x):=\frac{\psi(X_1+rX)}{r} \quad \text{for} \quad X\in B_{R/r}(0)$$
	with $B_r(X)\subset B_R(X_0)$. Then $\phi_r$ satisfies
	$$div( K(X_1+rX)\nabla\phi_r ) + rf(\rho\phi_r) \geq 0 \quad \text{in} \quad B_1.$$
	Denote $\epsilon:=\sup_{B_{\sqrt{\delta}}(0)}\phi_r$ for simplicity. It follows from Theorem 8.17 in \cite{GT77} that
	$$\epsilon\leq C\left[ (\fint_{B_1(0)}\phi_r^2 dX)^{1/2} + rF_0 \right] \leq C(c^*+rF_0),$$
	where $C=C(\kappa,\delta)$. Let the comparison function $\phi$ be defined as
	\begin{equation*}
	\begin{cases}
	div(K(X_1+rX)\nabla\phi)+rf(r\phi)=0 \quad \text{in} \quad B_{\sqrt{\delta}}(0)\backslash B_\delta(0), \\
	\phi=0 \qquad\qquad\qquad\qquad\qquad\qquad\quad\ \ \text{in} \quad B_\delta(0), \\
	\phi=\phi_r \qquad\qquad\qquad\qquad\qquad\qquad\quad\, \text{on} \quad \partial B_{\sqrt{\delta}}(0).
	\end{cases}
	\end{equation*}
	The elliptic estimate indicates that
	$$\sup_{\partial B_\delta(0)}|\nabla\phi|\leq C(X_0,\kappa,\delta,F_0)(\epsilon+rF_0).$$
	Meanwhile, set $\Phi=\min(\phi,\phi_r)$ and since $\psi$ is a local minimizer of $\mathcal{J}_h$ in $B_R(X_0)$, we obtain
	\begin{align*}
	&\int_{B_{\sqrt{\delta}}(0)} \left[ \left<K(X_1+rX)\nabla\phi_r,\nabla\phi_r\right>-2F(r\phi_r)+\lambda^2\chi_{\{\phi_r>0\}} \right]dX \\
	&\quad \leq \int_{B_{\sqrt{\delta}}(0)} \left[ \left<K(X_1+rX)\nabla\Phi,\nabla\Phi)\right>-2F(r\Phi)+\lambda^2\chi_{\{\phi>0\}} \right]dX
	\end{align*}
	and thus
	\begin{align*}
	&\int_{B_\delta(0)} \left[ \left<K(X_1+rX)\nabla\phi_r,\nabla\phi_r\right>-2F(r\phi_r)+\lambda^2\chi_{\{\phi_r>0\}} \right]dX \\
	&\quad \leq \int_{B_{\sqrt{\delta}}(0)\backslash B_\delta(0)} \left[ \left<K(X_1+rX)\nabla\Phi,\nabla\Phi\right>-\left<K(X_1+rX)\nabla\phi_r,\nabla\phi_r\right> \right. \\
	&\qquad\qquad \left. -2(F(r\Phi)-F(r\phi_r))+\lambda^2\chi_{\{\Phi>0\}}-\lambda^2\chi_{\{\phi_r>0\}} \right]dX \\
	&\quad \leq \int_{B_{\sqrt{\delta}}(0)\backslash B_\delta(0)} \left[ -\left<K(X_1+rX)\nabla(\Phi-\phi_r),\nabla(\Phi-\phi_r)\right> \right. \\
	&\qquad\qquad \left. +2\left<K(X_1+rX)\nabla\Phi,\nabla(\Phi-\phi_r)\right> -2(F(r\Phi)-F(r\phi_r)) \right] dX \\
	&\quad \leq \int_{B_{\sqrt{\delta}}(0)\backslash B_\delta(0)} \left[  2\left<K(X_1+rX)\nabla\Phi,\nabla(\Phi-\phi_r)\right>-2rf(r\Phi)(\Phi-\phi_r) \right]dX \\
	&\quad \leq 2\int_{\partial(B_{\sqrt{\delta}}(0)\backslash B_\delta(0))} (\phi-\phi_r)K(X_1+rX)\nabla\phi\cdot\nu d\mathcal{S} \\
	&\quad \leq C\int_{\partial B_\delta(0)} \phi_r|\nabla\phi|d\mathcal{S}.
	\end{align*}
	On the other hand, since $F(r\phi_r)\leq f(0)r\phi_r\leq F_0r\phi_r$, we can deduce that
	\begin{align*}
	& \int_{B_\delta(0)} \left( |\nabla\phi_r|^2 + \lambda^2\chi_{\{\phi_r>0\}} \right)dX \\
	&\leq C\int_{B_\delta(0)} \left[ \left<K(X_1+rX)\nabla\phi_r,\nabla\phi_r\right>-2F(r\phi_r)+\lambda^2\chi_{\{\phi_r>0\}} \right]dX + C\int_{B_\delta(0)} F_0r\phi_r dX \\
	&\leq C\left( \int_{\partial B_\delta(0)} \phi_r|\nabla\phi|d\mathcal{S} + \int_{B_\delta(0)}\phi_r dX \right) \\
	&\leq C(\epsilon+rF_0)\left( \int_{\partial B_\delta(0)}\phi_r d\mathcal{S} + \int_{B_\delta(0)} \phi_r dX \right).
	\end{align*}
	The subsequent proof follows similarly as in \cite{AC81}, utilizing the trace inequality and the H\"older inequality. More specifically,
	\begin{align*}
	\int_{\partial B_\delta(0)}\phi_r d\mathcal{S} + \int_{B_\delta(0)} \phi_r dX &\leq \int_{B_\delta(0)}\phi_r\chi_{\{\phi_r>0\}}dX + C\int_{B_\delta(0)}|\nabla\phi_r|\chi_{\{\phi_r>0\}}dX \\
	&\leq C\left( \frac{\epsilon+rF_0}{\lambda^2}+\frac{1}{\lambda} \right)\int_{B_\delta(0)} \left( |\nabla\phi_r|^2 + \lambda^2\chi_{\{\phi_r>0\}} \right)dX.
	\end{align*}
	Hence,
	$$\int_{B_\delta(0)} \left( |\nabla\phi_r|^2 + \lambda^2\chi_{\{\phi_r>0\}} \right)dX \leq C(\epsilon+rF_0)\left( \frac{\epsilon+rF_0}{\lambda^2}+\frac{1}{\lambda} \right)\int_{B_\delta(0)} \left( |\nabla\phi_r|^2 + \lambda^2\chi_{\{\phi_r>0\}} \right)dX,$$
	which indicates that
	$$\int_{B_\delta(0)} \left( |\nabla\phi_r|^2 + \lambda^2\chi_{\{\phi_r>0\}} \right)dX\equiv0$$
	provided that $\epsilon+rF_0$ is small enough, namely, $c^*$ is small enough. Consequently, $\phi_r\equiv0$ in $B_\delta(0)$, which leads to the desired result.
\end{proof}

\section{Regularity of the free boundary in a cross-section}

In this section we will demonstrate that the minimizer $\psi$ of $\mathcal{J}_h$ satisfies the criterion "flatness implies $C^{1,\alpha}$" through the blow-up analysis, thereby obtaining $C^{1,\alpha}$-regularity of the free boundary $\partial\{\psi>0\}$. Notably, while the rule "flatness implies $C^{1,\alpha}$" holds for our non-homogeneous elliptic operator as shown in \cite{S11}, it remains essential to verify that the minimizer $\psi$ possesses the "flatness" property. More specifically, we derive the explicit form of the blow-up limit at the free boundary point $X_0$ by way of freezing coefficients, obtaining a half-plane solution $\lambda\left( K^{-1/2}(X_0)X\cdot\nu_0 \right)^+$, which establishes the flatness improvement property of $\psi$.

The first tool in our analysis is an extension of the \emph{Weiss monotonicity formula}, originally developed in \cite{P83} and \cite{S84} for harmonic mappings. The \emph{Weiss boundary adjusted energy} in the Laplacian free boundary problem is defined as
\begin{align}
W(u,r)&= r^{-2}\int_{B_r} \left( |\nabla u|^2 +\lambda^2\chi_{\{u>0\}} \right)dX - r^{-3} \int_{\partial B_r} u^2d\mathcal{S},
\end{align}
which is related to the functional
\begin{equation}
\mathcal{J}_{\rm AC}(u,B_R)=\int_{B_R} \left( |\nabla u|^2 + \lambda^2\chi_{\{u>0\}} \right)dX,
\end{equation}
previously studied in \cite{AC81} and \cite{W99} with $r<R$ and $0\in\partial\{u>0\}$. We will perform a change of variables to freeze the coefficients and reduce to this case.

We start with the notation
\begin{equation*}
\mathcal{J}_h^{X_0,F}(\psi):= \int_{B_R(X_0)} \left[ \left<K(X_0)\nabla\psi,\nabla\psi\right> -2F(\psi) + \lambda^2\chi_{\{\psi>0\}} \right]dX.
\end{equation*}
Recall that $K(X)=(k_{ij}(X))_{2\times2}$ satisfies for some $C=C(\kappa,R)$ that
$$|k_{ij}(X)-k_{ij}(Y)|\leq C|X-Y| \quad \text{for} \quad i,j=1,2$$
and
$$\left<K(X)\nabla\psi,\nabla\psi\right>\geq C|\nabla\psi|^2 \quad \text{in} \quad B_R(X_0).$$
Direct calculation shows that if $\psi$ is a local minimizer of $\mathcal{J}_h$, then
\begin{align*}
\mathcal{J}_h^{X_0,F}(\psi) &= \mathcal{J}_h(\psi) + \int_{B_R(X_0)}\left<\left(K(X_0)-K(X)\right)\nabla\psi,\nabla\psi\right> dX \\
&\leq (1+CR)\mathcal{J}_h(\psi)
\end{align*}
for some $C=C(\kappa,R)$, and analogously $\mathcal{J}_h^{X_0,F}(\psi)\geq(1-C)R\mathcal{J}_h(\psi)$. Consequently, for any $\tilde{\psi}\in W^{1,2}(B_R(X_0))$ with $\tilde{\psi}=\psi$ on $\partial B_R(X_0)$, we obtain
$$\mathcal{J}_h^{X_0,F}(\psi)\leq \frac{1+CR}{1-CR}\mathcal{J}_h^{X_0,F}(\tilde{\psi}) \leq (1+C_0 R)\mathcal{J}_h^{X_0}(\tilde{\psi})$$
for some $C_0=(\kappa,R)$ when $R$ is small enough. This means that $\psi$ is locally an almost-minimizer of $\mathcal{J}_h^{X_0,F}$. Furthermore, noticing that
$$\int_{B_r(X_0)} -2F(\psi) dX\leq \int_{B_r(X_0)} -2\psi f(\psi)dX \leq \int_{B_r(X_0)} CF_0 r dX$$
in a small neighborhood $B_r(X_0)$ of the free boundary point $X_0$, we can deduce that $\psi$ is locally an almost-minimizer of the functional
$$\mathcal{J}_h^{X_0}(\psi):=\int_{B_r(X_0)} \left[ \left<K(X_0)\nabla\psi,\nabla\psi\right> + \lambda^2\chi_{\{\psi>0\}} \right]dX.$$

Now we transform the elliptic operator with fixed coefficient to the Laplacian operator through the change of variables. Set $Y=X_0+K^{1/2}(X_0)X$ with the matrix $K^{1/2}(X)$ satisfying $\left(K^{1/2}(X)\right)^2=K(X)$. The notation of $K^{1/2}(X)$ is well defined due to the matrix $K(X)$ being symmetric with positive eigenvalues. Define
\begin{equation}\label{u-1}
u(X)=\psi(X_0+K^{1/2}(X_0)X)=\psi(Y)
\end{equation}
for $X\in B:=\{X=(x,y)^T \ | \ X_0+K^{1/2}(X_0)X \in B_R(X_0)\}$. Hence,
\begin{align*}
\mathcal{J}_{\rm AC}(u) &=\int_{B} \left( |\nabla u|^2 + \lambda^2\chi_{\{u>0\}} \right)dX \\
&=\int_{B_R(X_0)} \left[ \left<K(X_0)\nabla\psi,\nabla\psi\right> + \lambda^2\chi_{\{\psi>0\}} \right]\left|\det K^{-1/2}(X_0)\right|dY \\
&=\left|\det K^{-1/2}(X_0)\right|\mathcal{J}_h^{X_0}(\psi),
\end{align*}
which indicates that $u$ is a local almost minimizer of the functional $\mathcal{J}_{\rm AC}$ in $B_R(X_0)$.

Notice that $u$ keeps the properties of Lipschitz regularity and non-degeneracy. A standard blow-up argument for the almost minimizer $u$ gives the following properties, and the readers can refer to \cite{DET19} or \cite{STV20} for more details.

\begin{lem} \label{u}
Let $u$ be the local almost minimizer of $\mathcal{J}_{\rm AC}$ defined as in (\ref{u-1}). Then the following properties hold.

(i) The right limit $W(u,0+):=\lim_{r\rightarrow0+}W(u,r)$ exists and is finite.

(ii) The blow-up limit
$$u_0:=\lim_{r_k\rightarrow0} u_k = \lim_{r_k\rightarrow0} \frac{u(r_kX)}{r_k}$$
is a $1$-homogeneous function, and the convergence $u_k\rightarrow u_0$ is strong in $W_{\rm loc}^{1,2}(\mathbb{R}^2)$ and locally uniform in $\mathbb{R}^2$ under unrelabeled subsequence.

(iii) There is a unique unit vector $\nu_0\in\partial B_1$ such that $u_0 = \lambda(X\cdot\nu_0)^+=\lambda\max(X\cdot\nu_0,0)$ is a half-plane solution.
\end{lem}

Now we come back to the blow-up sequence for $\psi$ at $X_0$:
\begin{equation} \label{blsq}
\psi_k(X):=\frac{\psi(X_0+r_kX)}{r_k}
\end{equation}
and compute the expicit form of the blow-up limit
$$\psi_0:=\lim_{r_k\rightarrow0}\psi_k.$$

\begin{lem} \label{psi}
	Suppose $\psi$ is a local minimizer of $\mathcal{J}_h$ in $B_R(X_0)$, where $X_0\in\partial\{\psi>0\}$ is a free boundary point. Suppose that the blow-up sequence $\psi_k(X)$ at $X_0$ defined as in (\ref{blsq}) converges to $\psi_0(X)$ weakly in $W_{\rm loc}^{1,2}(\mathbb{R}^2)$. Then
	$$\psi_k(X)\rightarrow\psi_0(X) \quad \text{strongly in} \quad W_{\rm loc}^{1,2}(\mathbb{R}^2).$$
	Moreover, $\psi_0(X)$ is a one-homogeneous function with the form
	$$\psi_0(X)=\lambda\left( K^{-1/2}(X_0)X\cdot\nu_0 \right)^+.$$
\end{lem}

\begin{proof}
To begin with, notice that $\psi_k$ is a local minimizer of
$$\mathcal{J}_{k,\rm h}(\psi_k)=\int_{B_R/r_k(0)} \left( \left<K(X_0+r_kX)\nabla\psi_k,\nabla\psi_k\right> -2F(r_k\psi_k) + \lambda^2\chi_{\{\psi_k>0\}} \right)dX.$$
Due to the fact that $\psi_k\in C^{0,1}(B_R/r_k(0))$, the uniform convergence from $\psi_k$ to $\psi_0$ indicates that
$$div (K(X_0)\nabla\psi_0)=0 \quad \text{in} \quad \{\psi_0>0\}.$$
Thus we obtain that for any $\eta\in W_0^{1,2}(B)$ with $B\subset\mathbb{R}^2$ and any small enough $r_k$,
\begin{align*}
&\int_{B} \left<K(X_0+r_kX)\nabla\psi_k,\nabla\psi_k\right>\eta dX \\
&=-\int_{B} \psi_k div(K(X_0+r_kX)\nabla\psi_k)\eta dX - \int_{B} \psi_k\left<K(X_0+r_kX)\nabla\psi_k,\nabla\eta\right> dX \\
&\rightarrow -\int_{B} \psi_0 div(K(X_0)\nabla\psi_0)\eta dX - \int_{B} \psi_0\left<K(X_0)\nabla\psi_0,\nabla\eta\right> dX \\
&=\int_{B} \left<K(X_0)\nabla\psi_0,\nabla\psi_0\right>\eta dX,
\end{align*}
which implies that $\nabla\psi_k\rightarrow\nabla\psi_0$ strongly in $W_{\rm loc}^{1,2}(\mathbb{R}^2)$. This finishes the proof that $\psi_k\rightarrow\psi_0$ strongly in $W_{\rm loc}^{1,2}(\mathbb{R}^2)$.

The remaining part of this lemma follows directly from the observation that
\begin{align*}
\psi_k(X) &= \frac{\psi(X_0+r_kX)}{r_k} \\ &=\frac{\psi\left(X_0+K^{1/2}(X_0)K^{-1/2}(X_0)(r_kX)\right)}{r_k} \\
&=\frac{u\left( r_k K^{-1/2}(X_0)X \right)}{r_k} \\
&= u_k \left( K^{-1/2}(X_0)X \right),
\end{align*}
which indicates that
$$\psi_0(X)=u_0 \left( K^{-1/2}(X_0)X \right).$$
\end{proof}

Lemma \ref{psi} combining with Lemma \ref{u}, (iii)
implies that for any small $\epsilon>0$, there exists $r_k$ small enough such that
$$\lambda\left( K^{-1/2}(X_0)X\cdot\nu_0-\epsilon \right)^+ \leq \psi_k(X) \leq \lambda\left( K^{-1/2}(X_0)X\cdot\nu_0+\epsilon \right)^+.$$

We are now ready to utilize the rule of "flatness implies $C^{1,\alpha}$" for the uniform elliptic operator with nontrivial right hand side.

\begin{lem}[Theorem 1.1 in \cite{S11}] \label{rule}
	Let $v$ be a viscosity solution to
	\begin{equation*}
	\begin{cases}
	\Sigma_{i,j} a_{ij}(X)\partial_{ij} v = f(X) \quad \text{in} \quad \Omega\cap\{v>0\}, \\
	|\nabla v|=\lambda \qquad\qquad\qquad\ \ \text{on} \quad \Omega\cap\partial\{u>0\},
	\end{cases}
	\end{equation*}
	where $\partial_{ij}v$ denotes the second partial derivative with respect to the $i,j$ component and $\Omega$ is a bounded domain, $a_{ij}\in C^{0,\beta}(\Omega)$ and $f\in C(\Omega)\cap L^\infty(\Omega)$. Assume that $0\in\partial\{u>0\}$ and $\Vert a_{ij}-\delta_{ij}\Vert_{L^\infty(B_1)}\leq \epsilon_0$ for some uniform $\epsilon_0$ small enough. There exists a universal constant $\bar{\epsilon}>0$ such that if
	$$\lambda(X\cdot\nu_0-\bar{\epsilon})^+ \leq v(X) \leq \lambda(X\cdot\nu_0+\bar{\epsilon})^+ \quad \text{in} \quad B_1$$
	for some $\nu_0\in\partial B_1$ and
	$$[a_{ij}]_{C^{0,\beta}(B_1)}\leq\bar{\epsilon}, \quad \Vert f \Vert_{L^\infty(B_1)}\leq\bar{\epsilon},$$
	then the free boundary $\partial\{u>0\}$ is $C^{1,\alpha}$ in $B_{1/2}$ for some $0<\alpha<1$.
\end{lem}

The proof of Lemma \ref{rule} involves two key ingredients, the partial boundary Harnack inequality and the analysis of the linearized problem. This groundbreaking method was proposed by De Silva in 2011 \cite{S11}, which could be applied in various kinds of problems, such as two-phase Bernoulli problems and shape optimization problems, providing a different way of proof about the rule "flatness implies $C^{1,\alpha}$" from the seminal works by Caffarelli in \cite{C87},\cite{C88} and \cite{C89}. Note that the subsolution method by Caffarelli is also applicable in our case, which is beyond the scope of this paper.

\begin{proof}[Proof of Theorem B]
	Lemma \ref{rule} gives the criterion of the regularity of the free boundary, and it remains to check the conditions. Suppose $X_0=(x_0,y_0)\in\partial\{\psi>0\}.$ Translate $X_0$ to $0$ and denote for simplicity that $\psi(x,y)=\psi(x-x_0,y-y_0)$, then the blow-up sequence $\psi_k$ minimizes
	$$\mathcal{J}_k (\psi_k) = \int_{B_R/r_k(0)} \left[ \left<K(r_kX)\nabla\psi_k,\nabla\psi_k\right> -2F(r_k\psi_k) + \lambda^2\chi_{\{\psi_k>0\}} \right]dX,$$
	and $\psi_k$ is the viscosity solution (see Chapter 7 in \cite{V23}) to the free boundary problem
	\begin{equation*}
	\begin{cases}
	div\left(K(r_kX)\nabla\psi_k\right) = -r_kf(r_k\psi_k) \quad \text{in} \quad \{\psi_k>0\}\cap B_R/r_k(0), \\
	\left<K(r_kX)\nabla\psi_k,\nabla\psi_k\right>=\lambda^2 \qquad\qquad\, \text{on} \quad \partial\{\psi_k>0\}\cap B_R/r_k(0).
	\end{cases}
	\end{equation*}
	Denote $K_k(X):=K(r_kX)=(a_{ij}(X))_{2\times2}$ for $i,j=1,2$ with
	\begin{equation*}
	K_k(X)= \frac{1}{\kappa^2+(r_kx)^2+(r_ky)^2}
	\begin{pmatrix}
	\kappa^2+(r_ky)^2 & -(r_k)^2xy \\
	-(r_k)^2xy & \kappa^2+(r_kx)^2 \\
	\end{pmatrix}.
	\end{equation*}
	We can compute that
	\begin{equation*}
	K_k(X)-
	\begin{pmatrix}
	1 & 0 \\
	0 & 1 \\
	\end{pmatrix}
	=\frac{r_k^2}{\kappa^2+r_k^2(x^2+y^2)}
	\begin{pmatrix}
	-x^2 & -xy \\
	-xy & -y^2
	\end{pmatrix},
	\end{equation*}
	which indicates that for any $\epsilon>0$ there exists $r_k=r_k(\epsilon)$ small enough such that
	$$\Vert a_{ij}-\delta_{ij}\Vert_{L^\infty(B_1)}\leq Cr_k^2 \leq \epsilon$$
	and
	$$[a_{ij}]_{C^{0,\beta}(B_1)}\leq\epsilon.$$
	On the other hand, since $f(s)\in C^{1,\beta}(\mathbb{R})$ and $\psi\in C^{0,1}(B_R(X_0))$, the composite function $f(\psi(X))$ must be continuous in $B_R(X_0)$. Moreover, the function $r_kf(r_k\psi_k)$ is continuous in $B_1$, and satisfies
	$$|r_kf(r_k\psi_k)|\leq r_kF_0\frac{r_kF_0}{r_k}\leq (F_0)^2 r_k\leq\epsilon.$$
	Hence there exists a universal $\bar{\epsilon}>0$ such that for $r_k$ small enough,
	$$\lambda\left( K^{-1/2}(X_0)X\cdot\nu_0-\epsilon \right)^+ \leq \psi_k(X) \leq \lambda\left( K^{-1/2}(X_0)X\cdot\nu_0+\epsilon \right)^+$$
	and $\psi_k$ satisfies all the conditions in Lemma \ref{rule}. Consequently, $\partial\{\psi_k>0\}$ is $C^{1,\alpha}$ in $B_{R/(2r_k)}(0)$, which yields that $\partial\{\psi>0\}$ is $C^{1,\alpha}$ in $B_{R/2}(0)$. Such translation can be carried on for any free boundary point $X_0\in\partial\{\psi>0\}$, thus $\partial\{\psi>0\}$ is locally $C^{1,\alpha}$. This completes the proof.
\end{proof}

\appendix

\section{The field of tangents of the symmetry lines}\label{appe1}

In Appendix \ref{appe1} we verify the rules in Lemma \ref{lem1.1} that helical function $f$ and helical vector field $\bm v$ obey with $\bm\xi_\kappa$, which represents the tangent of the symmetry line of the group $G^\kappa$.

\begin{proof}[Proof of Lemma \ref{lem1.1}]
	Take derivatives with respect to $\rho$ on both sides of $f(S_\rho^\kappa\bm x) = f(\bm x)$ and then take the value at $\rho=0$, we have
	$$y\frac{\partial f}{\partial x} - x\frac{\partial f}{\partial y} + \kappa\frac{\partial f}{\partial z}=0,$$
	which yields that $\nabla f\cdot\bm\xi_\kappa=0$. To prove the converse of this statement, denote $g(\rho):=f(S_\rho^\kappa\bm x)$ for any fixed $\kappa\neq0$ and $\bm x=(x,y,z)\in\mathbb{R}^3$. Then $g(\rho)$ satisfies the following differential equation
	\begin{equation*}
	\begin{aligned}
	\frac{dg(\rho)}{d\rho} &=(-x\sin\rho + y\cos\rho)\frac{\partial f}{\partial x} + (-x\cos\rho-y\sin\rho)\frac{\partial f}{\partial y} + \kappa\frac{\partial f}{\partial z} \\
	&=\nabla f(S_\rho^\kappa\bm x)\cdot\bm\xi_\kappa(S_\rho^\kappa\bm x) \\
	&=0,
	\end{aligned}
	\end{equation*}
	where we have used the relationship $\nabla f\cdot\bm\xi_\kappa=0$ in the last equality. Therefore,
	$$g(\rho)=g(0)=f(\bm x),$$
	which gives the desired result.
	
	Now we prove the claim for the differentiable vector field $\bm v$. Take derivatives with respect to $\rho$ on both sides of $\bm v(S_\rho^\kappa \bm x) = R_\rho \bm v (\bm x)$ at $\rho=0$, we get that
	$$\frac{d v_i(S_\rho^\kappa\bm x)}{d\rho}\Bigg|_{\rho=0} = y\frac{\partial v_i}{\partial x} -x\frac{\partial v_i}{\partial y} +\kappa\frac{\partial v_i}{\partial z} = \nabla v_i\cdot\bm\xi_\kappa \quad \text{for} \quad i=1,2,3,$$
	and
	$$\frac{d R_\rho}{d\rho} \Bigg|_{\rho=0}=
	\begin{pmatrix}
	0 & 1 & \\
	-1 & 0 & \\
	& & 0 \\
	\end{pmatrix}.$$
	Hence,
	$$\nabla v_1 \cdot\bm\xi_\kappa = v_2, \
	\nabla v_2 \cdot\bm\xi_\kappa = -v_1, \
	\nabla v_3 \cdot\bm\xi_\kappa = 0.$$
	The proof of the converse statement follows in an analogous way and we omit it here.
\end{proof}

\section{Basic properties of minimizers}\label{appe2}

In Appendix \ref{appe2} we show the fact that the local minimizer $\psi$ of $\mathcal{J}_h$ is a solution to the free boundary problem (\ref{eq1}) in $B_R(X_0)$.

\begin{lem}
	Suppose $\psi$ is a local minimizer of $\mathcal{J}_h$ in $B_R(X_0)\Subset D$. Then $\psi$ satisfies
	\begin{equation*}
	\begin{cases}
	div (K(x,y)\nabla\psi) = f(\psi) \qquad\qquad\quad \text{in} \quad \{\psi>0\}\cap B_R(X_0), \\
	\left<K(x,y)\nabla\psi,\nabla\psi\right> = \lambda^2 \qquad\qquad\ \text{on} \quad \partial\{\psi>0\}\cap B_R(X_0).
	\end{cases}
	\end{equation*}
\end{lem}

\begin{proof}
	The minimality of $\psi$ gives that for any non-negative function $\phi\in W_0^{1,2}(\{\psi>0\}\cap B_R(X_0))$, we have
	$$\mathcal{J}_h(\psi,B_R(X_0))\leq\mathcal{J}_h(\psi+\epsilon\phi,B_R(X_0)) \quad \text{for} \quad \epsilon\in\mathbb{R},$$
	which indicates
	$$\int_{B_R(X_0)} \left( \left<K(x,y)\nabla\psi,\nabla\phi\right> -f(\psi)\phi \right)dX=0$$
	if we take $\epsilon\rightarrow0$. Hence $\psi$ satisfies the desired elliptic equation in $\{\psi>0\}\cap B_R(X_0)$.
	
	To verify the free boundary condition on $\partial\{\psi>0\}\cap B_R(X_0)$, denote $(z_1,z_2)^T:=\bm\Phi_t(x,y)=(x,y)^T+t\bm\xi(x,y)$ with $\bm\xi=(\xi_1,\xi_2) \in W_0^{1,2}(B_R(X_0))$ and $t\in\mathbb{R}$. Let $\phi_t(z_1,z_2)=\phi(x,y)$. The domain variation gives
	\begin{align*}
	0 &= \frac{d}{dt}\Bigg|_{t=0}\mathcal{J}_h(\phi_t,B_R(X_0)) \\
	&=\frac{d}{dt}\Bigg|_{t=0}\int_{B_R(X_0)} \left[ \left<K(\Phi_t(x,y))\nabla\psi,\nabla\psi\right> - 2t\left<K(\Phi_t(x,y))\nabla\psi, D\bm\xi\nabla\psi\right> + o(t) \right. \\
	& \qquad\left. -2F(\psi) +\lambda^2\chi_{\{\psi>0\}} \right] (1+t div\bm\xi +o(t)) dX \\
	&=\int_{B_R(X_0)} \left[ \left<K(x,y)\nabla\psi,\nabla\psi\right>div\bm\xi + \left<(\bm\xi\cdot\nabla)K(x,y)\nabla\psi,\nabla\psi\right> \right. \\
	&\qquad \left. -2\left<K(x,y)\nabla\psi,D\bm\xi\nabla\psi\right> -2F(\psi)div\bm\xi +\lambda^2\chi_{\{\psi>0\}}div\bm\xi \right] dX \\
	&= \int_{B_R(X_0)\cap\{\psi>0\}} div\left[ \left<K(x,y)\nabla\psi,\nabla\psi\right>\bm\xi - 2K(x,y)\nabla\psi(\bm\xi\cdot\nabla\psi) + \lambda^2\bm\xi \right]dX \\
	&=\int_{B_R(X_0)\cap\partial\{\psi>0\}} [-\left<K(x,y)\nabla\psi,\nabla\psi\right>+\lambda^2](\bm\xi\cdot\nu) d\mathcal{S}.
	\end{align*}
	Notice that we have used the fact that $K(x,y)$ is symmetric. This completes the proof.
\end{proof}

\subsubsection*{Acknowledgment}
This work is supported by National Nature Science Foundation of China under Grant 12125102, and Nature Science Foundation of Guangdong Province under Grant 2024A1515012794, and Shenzhen Science and Technology Program JCYJ2024120 2124209011.
\subsubsection*{Declaration of competing interest} The authors declare they have no conflicts of interests. 
\subsubsection*{Data availability} No data was used for the research described in the article.

\vskip .4in

\bibliographystyle{plain}
\bibliography{helical.bib}

\end{document}